\documentclass[a4paper,12pt]{article}
\usepackage{mathrsfs}
\usepackage{algorithm, algpseudocode}
\usepackage{amsmath,amssymb,amsthm,latexsym,amsfonts}
\usepackage{pstricks}
\usepackage{enumerate}
\usepackage{epsfig}
\usepackage{epstopdf}
\usepackage{graphicx}
\usepackage{color}
\usepackage{ifpdf}
\usepackage{caption}

\begin{document}

\newtheorem{lem}{Lemma}
\newtheorem{thm}{Theorem}
\newtheorem{cor}{Corollary}
\newtheorem{exa}{Example}
\newtheorem{con}{Conjecture}
\newtheorem{rem}{Remark}
\newtheorem{obs}{Observation}
\newtheorem{definition}{Definition}
\newtheorem{pro}{Proposition}
\theoremstyle{plain}
\newcommand{\D}{\displaystyle}
\newcommand{\DF}[2]{\D\frac{#1}{#2}}

\renewcommand{\figurename}{{\bf Fig}}
\captionsetup{labelfont=bf}

\title{\bf Conflict-free connection numbers\\[3mm] of line graphs\footnote{Supported by NSFC No.11371205 and 11531011, and CPSF No. 2016M601247.}}

\author{{\small Bo Deng$^{1,2}$, \ \  Wenjing Li$^1$, \ \
Xueliang Li$^{1,2}$, \ \ Yaping Mao$^2$,  \ \ Haixing Zhao$^2$}\\
\\{\small $^1$Center for Combinatorics and LPMC,}\\
 {\small Nankai University, Tianjin 300071, China.}\\
      {\small $^2$School of Mathematics and Statistics,}\\
      {\small Qinghai Normal University,}
         \\{\small  Xining, Qinghai, 810008, China.}\\
      {\small Emall: dengbo450@163.com;}\\
        {\small liwenjing610@mail.nankai.edu.cn;}\\
         {\small lxl@nankai.edu.cn;}\
           {\small maoyaping@ymail.com;}\\
         {\small h.x.zhao@163.com.}
         }

\date{}

\maketitle
\begin{abstract}
A path in an edge-colored graph is called \emph{conflict-free}
if it contains at least one color used on exactly one of its edges.
An edge-colored graph $G$ is \emph{conflict-free connected}
if for any two distinct vertices of $G$, there is a conflict-free path connecting them.
For a connected graph $G$, the \emph{conflict-free connection number} of $G$, denoted by $cfc(G)$,
is defined as the minimum number of colors that are required to make $G$
conflict-free connected.
In this paper, we investigate the conflict-free connection numbers of
connected claw-free graphs, especially line graphs.
We first show that for an arbitrary connected graph $G$, there exists a positive integer $k$ such that $cfc(L^k(G))\leq 2$.
Secondly, we get the exact value of the conflict-free connection number of a connected claw-free graph, especially a connected line graph.
Thirdly, we prove that for an arbitrary connected graph $G$
and an arbitrary positive integer $k$, we always have $cfc(L^{k+1}(G))\leq cfc(L^k(G))$,
with only the exception that $G$ is isomorphic to a star of order at least~$5$ and $k=1$.
Finally, we obtain the exact values
of $cfc(L^k(G))$, and use them as an efficient tool to get the smallest
nonnegative integer $k_0$ such that $cfc(L^{k_0}(G))=2$.\\[3mm]


\end{abstract}

\section{Introduction}

All graphs considered in this paper are simple, finite, and undirected.
We follow the terminology and notation of Bondy and Murty in~\cite{Bondy} for those not defined here.
For a connected graph $G$, let $V(G),E(G),\kappa(G)$ and $\lambda(G)$ denote the vertex set, the edge set, the vertex-connectivity and the edge-connectivity of $G$, respectively.
Throughout this paper, we use $P_n$, $C_n$ and $K_n$ to denote a path, a cycle and a complete graph of order $n$, respectively.
And we call $G$ a \emph{star} of order $r+1$ if $G\cong K_{1,r}$.

Let $G$ be a nontrivial connected graph with an \emph{edge-coloring c}
$:E(G)\rightarrow \{0,1,\dots,t\}$, $t\in \mathbb{N}$,
where adjacent edges may be colored with the same color.
When adjacent edges of $G$ receive different colors by $c$, the edge-coloring $c$ is called \emph{proper}.
The \emph{chromatic index} of $G$, denoted by $\chi'(G)$, is
the minimum number of colors needed in a proper coloring of $G$.
A path in $G$ is called \emph{a rainbow path} if no two edges
of the path are colored with the same color. The graph $G$ is called \emph{rainbow connected}
if for any two distinct vertices of $G$,
there is a rainbow path connecting them.
For a connected graph $G$, the \emph{rainbow connection number} of $G$, denoted by $rc(G)$, is defined as the
minimum number of colors that are needed to make $G$ rainbow connected.
These concepts
were first introduced by Chartrand et al. in~\cite{Char1} and have been well-studied since then.
For further details, we refer the reader to a book~\cite{Li} and a survey paper \cite{LSS}.

Motivated by the rainbow connection coloring and proper coloring in graphs, Andrews et al.~\cite{Andrews} and Borozan et al.~\cite{Magnant}
proposed the concept of proper-path coloring.
Let $G$ be a nontrivial connected graph with an edge-coloring.
A path in $G$ is called \emph{a proper path}
if no two adjacent edges of the path are colored with the same color.
The graph $G$ is called \emph{proper connected}
if for any two distinct vertices of $G$,
there is a proper path connecting them.
The \emph{proper connection number} of $G$,
denoted by $pc(G)$, is defined as the minimum number of colors
that are needed to make $G$ proper connected.
For more details, we refer to a dynamic survey~\cite{Li1}.

Inspired by the above mentioned two connection colorings and
conflict-free colorings of graphs and hypergraphs~\cite{Pach},
Czap et al.~\cite{Czap} recently introduced the concept of
the conflict-free connection number of a nontrivial connected graph.
Let $G$ be a nontrivial connected graph with an edge-coloring $c$.
A path in $G$ is called \emph{conflict-free}
if it contains at least one color used on exactly one of its edges.
The graph $G$ is \emph{conflict-free connected (with respect to the edge-coloring $c$)}
if for any two distinct vertices of $G$, there is a conflict-free path connecting them.
In this case, the edge-coloring $c$ is called a \emph{conflict-free connection coloring} (\emph{$CFC$-coloring} for short).
For a connected graph $G$, the \emph{conflict-free connection number}
of $G$, denoted by $cfc(G)$,
is defined as the minimum number of colors that are required to make $G$
conflict-free connected. For the graph with a single vertex or without any vertex, we assume the value of its conflict-free connection number equal to $0$.
The following observations are immediate.

\begin{pro}\label{pro1} Let $G$ be a connected graph on $n\geq 2$ vertices. Then we have

\ \ $(i)$ $cfc(G)=1$ if and only if $G$ is complete;

\ $(ii)$ $cfc(G)\geq 2$ if $G$ is noncomplete;

$(iii)$ $cfc(G)\leq n-1$.
\end{pro}

In~\cite{Czap}, Czap et. al first gave the exact value of the
conflict-free connection number for a path.

\begin{thm}[\cite{Czap}]\label{thm1}
$cfc(P_n)=\lceil log_2 n\rceil$.
\end{thm}

Then they
investigated the graphs
with conflict-free connection number~$2$.
Let $C(G)$ be the subgraph of $G$ induced
by the set of cut-edges of $G$.

\begin{thm}[\cite{Czap}]\label{thm2}
If $G$ is a noncomplete~$2$-connected graph, then $cfc(G)=2$.
\end{thm}

\begin{thm}[\cite{Czap}]\label{thm3}
If $G$ is a connected graph with at least~$3$ vertices and $C(G)$ is a linear forest
whose each component is of order~$2$, then $cfc(G)=2$.
\end{thm}

In fact, we can weaken the condition of Theorem~\ref{thm2},
and get that the same result holds for 2-edge-connected graphs, whose proof is similar to that
of Theorem~\ref{thm3} in \cite{Czap}. For completeness, we give its
proof here. Before we proceed to the result and its proof, we need the following lemmas which are useful in our proof, and can be found in~\cite{Czap}.

\begin{lem}[\cite{Czap}]\label{lem1}
Let $u,v$ be two distinct vertices and let $e=xy$
be an edge of a~$2$-connected graph $G$.
Then there is a $u$ -- $v$ path in $G$ containing the edge $e$.
\end{lem}

A \emph{block} of a graph $G$ is a maximal connected subgraph
of $G$ without cut-vertices.
A connected graph with no cut-vertex therefore has just one block,
namely the graph itself.
An edge is a block if and only if it is a cut-edge.
A block consisting of an edge is called trivial.
Note that any nontrivial block is~$2$-connected.

\begin{lem}[\cite{Czap}]\label{lem2}
Let $G$ be a connected graph. Then from its every nontrivial
block an edge can be chosen so that the set of all such chosen edges
forms a matching.
\end{lem}

\begin{thm}\label{thm4}
Let $G$ be a noncomplete $2$-edge-connected graph.
Then $cfc(G)=2$.
\end{thm}
\begin{proof}
If $G$ is a noncomplete $2$-connected graph,
then we are done.
So we only consider the case that $G$ has
at least one cut-vertex.
Note that
$G$ has a block decomposition with each block having at least~$3$
vertices, that is, each block is nontrivial.
By Lemma~$2$, we choose from each block one edge so that all
chosen edges create a matching $S$.
Next we color the edges from $S$ with color~$2$
and all remaining edges of $G$ with color~$1$.

Now we prove this coloring makes $G$ conflict-free connected,
that is, for any two distinct vertices $x$ and $y$,
we need to find a conflict-free $x$ -- $y$ path.

{\bf Case~$1$.} Let $x$ and $y$ belong to the same block $B$.
Then by Lemma~$1$, there is an $x$ -- $y$ path, in $B$,
containing the edge of $B$ colored with color~$2$.
Clearly, this $x$ -- $y$ path is conflict-free.

{\bf Case~$2$.} Let $x$ and $y$ be in different blocks.
Consider a shortest $x$ -- $y$ path in $G$.
This path goes through blocks, say $B_1,B_2,\dots,B_r$, $r\geq 2$,
in this order, where $x\in V(B_1)$ and $y\in V(B_r)$.
Let $v_i$ be a common vertex of blocks $B_i$ and $B_{i+1}$,
$1\leq i\leq r-1$. Set $y=v_r$. Clearly, $x\neq v_1$.
We choose an $x$ -- $v_1$ path in $B_1$ going through the edge assigned~$2$,
and then a $v_i$ -- $v_{i+1}$ path in $B_i$ omitting the edge
colored with~$2$ in $B_i$ for $2\leq i\leq r$.
Obviously, the concatenation of the above $r$ paths is
an conflict-free $x$ -- $y$ path.
\end{proof}

For a general graph $G$ with connectivity 1,
the authors of \cite{Czap} gave the bounds on $cfc(G)$.
Let $G$ be a connected graph
and $h(G)=max\{cfc(K): K$ is a component of $C(G)\}$.
In fact, $h(G)=0$ if $G$ is~$2$-edge-connected.
So we restate that theorem as follows.

\begin{thm}[\cite{Czap}]\label{thm5}
If $G$ is a connected graph with at least one cut-edge,
then $h(G)\leq cfc(G)\leq h(G)+1$. Moreover, these bounds
are tight.
\end{thm}

Line graphs form one of the most important graph class, and there have been a lot of results on line graphs, see \cite{Line}. In this paper we also deal with line graphs. Recall that the \emph{line graph} of a graph $G$ is the graph $L(G)$
whose vertex set $V(L(G))=E(G)$ and two vertices
$e_1,e_2$ of $L(G)$ are adjacent if and only if
they are adjacent in $G$.
The \emph{iterated line graph}
of a graph $G$, denoted by $L^2(G)$,
is the line graph of the graph $L(G)$.
In general, the \emph{$k$-iterated line graph}
of a graph $G$, denoted by $L^k(G)$,
is the line graph of the graph $L^{k-1}(G)$,
where $k\geq 2$ is a positive integer.
We call a graph \emph{claw-free} if it does not contain
a claw $K_{1,3}$ as its induced subgraph.
Notice that a line graph is claw-free; see~\cite{Beineke}
or \cite{Line}.

This paper is organized as follows: In Section~$2$,
we give some properties concerning the line graphs,
and based on them, we show that for an arbitrary connected graph $G$, there exists a positive integer $k$ such that $cfc(L^k(G))\leq 2$.
In Section~$3$,
we start with the investigation of one special family of graphs,
and then classify the graphs among them with $cfc(G)=h(G)+1$.
Using this result, we first get the exact value of the conflict-free connection number of a connected claw-free graph.
As a corollary, for a connected line graph $G$,
we obtain the value of $cfc(G)$.
Then, we prove that for an arbitrary connected graph $G$
and an arbitrary positive integer $k$, we always have $cfc(L^{k+1}(G))\leq cfc(L^k(G))$,
with only the exception that $G$ is isomorphic to a star of order at least~$5$ and $k=1$.
Finally, we obtain the exact values
of $cfc(L^k(G))$, and use them as an efficient tool to get the smallest
nonnegative integer $k_0$ such that $cfc(L^{k_0}(G))=2$.

\section{Dynamic behavior of the line graph operator}

If one component $\mathcal{C}$ of $C(G)$ is either a cut-edge
or a path of order at least~$3$ whose internal vertices are
all of degree~$2$ in $G$, then we call $\mathcal{C}$
\emph{a cut-path} of $G$.

\begin{lem}\label{lem4}
For a connected claw-free graph $G$,
each component
of $C(G)$ is a cut-path of $G$.
\end{lem}
\begin{proof}
Firstly, $C(G)$ is a linear forest.
Otherwise, there exists a vertex $v\in C(G)$
whose degree is larger than~$2$ in $C(G)$.
Then $v$ and three neighbors of $v$ in $C(G)$
induce a $K_{1,3}$ in $G$,
contradicting the condition that $G$ is claw-free.
Secondly, with a similar reason, if one component
of $C(G)$ has at least~$3$ vertices,
then all of its internal vertices must be
of degree~$2$ in $G$. So, each component of $C(G)$
must be a cut-path of $G$.
\end{proof}

Since a line graph is claw-free,
Lemma~\ref{lem4} is valid for line graphs.

\begin{cor}\label{cor1}
For a connected line graph $G$,
every component
of $C(G)$ is a cut-path of $G$.
\end{cor}

In~$1969$, Chartrand and Stewart~\cite{Char2}
showed
that $\kappa(L(G))\geq \lambda(G)$,
if $\lambda(G)\geq2$.
So, the following result is obvious.

\begin{lem}\label{lem3}
The line graph of a~$2$-edge-connected graph is~$2$-connected.
\end{lem}

Now, we examine the dynamic behavior of the line graph operator, and
get our main result of this section.

\begin{thm}\label{thm6}
For any connected graph $G$, there exists a positive integer
$k$ such that $cfc(L^k(G))\leq 2$.
\end{thm}

\begin{proof} If $G$ is a~$2$-edge-connected graph,
then by Proposition~\ref{pro1}, Theorem~\ref{thm2} and Lemma~\ref{lem3},
we obtain $cfc(L(G))\leq 2$.
In this case, we set $k=1$.
In the following, we concentrate on the graphs having
at least one cut-edge.

Let $\mathcal{P}$ be a set of paths in $C(G)$ who have at least one
internal vertex and whose internal vertices are all of degree~$2$ in $G$.
If $\mathcal{P}=\emptyset$, then $L(G)$ is~$2$-edge-connected.
Otherwise, if there is a cut-edge $e_1e_2$ in $L(G)$, then
there is a path of length~$2$ in $G$ whose internal vertex is of degree~$2$, which is
a contradiction. Thus, by Proposition~\ref{pro1} and Theorem~\ref{thm4},
we have $cfc(L(G))\leq 2$. Then we also set $k=1$ in this case.

If $\mathcal{P}\neq\emptyset$, let $p$ be the length of a longest path among $\mathcal{P}$.
Then, by Corollary~\ref{cor1}, each component of $C(L^i(G))$
must be a cut-path of $L^i(G)$ for $1\leq i \leq p$.
Since $L(P_j)=P_{j-1}$ for any positive integer $j\geq 1$,
each component of $C(L^{p-1}(G))$ is of order~$2$.
By Theorem~\ref{thm3}, we have $cfc(L^{p-1}(G))=2$.
Thus, we set $k=p-1$ in this case.

The proof is thus complete.
\end{proof}

\section{The values $cfc(L^{k}(G))$ of iterated line graphs}

In this section, we first investigate the connected graphs $G$
having at least one cut-edge and each component of
$C(G)$ is a cut-path of $G$.
Among them, we characterize the graphs $G$ satisfying $cfc(G)=h(G)$,
and the graphs $G$ satisfying $cfc(G)=h(G)+1$.
Let $G$ be a connected graph of order $n$.
If $n=2$, $G\cong P_2$,
and hence $cfc(G)=h(G)=1$.
In the following, we assume $n\geq 3$.
If $h(G)=1$,
then by Theorem~\ref{thm3}, we always
have $cfc(G)=2=h(G)+1$.
So we only need to discuss the case of $h(G)\geq 2$.

\begin{thm}\label{thm7}
Let $G$ be a connected graph having at least one cut-edge,
and $C(G)$ be its linear forest whose each component is
a cut-path of $G$ and $h(G)\geq 2$.
Then $cfc(G)=h(G)+1$ if and only if there are at least
two components of $C(G)$ whose conflict-free connection numbers
attain $h(G)$; and $cfc(G)=h(G)$ if and only if there is only one component of $C(G)$ whose conflict-free connection number
attains $h(G)$.
\end{thm}

\begin{proof} We first consider the case that
there are at least two components of $C(G)$
whose conflict-free connection numbers
attain $h(G)$, say $\mathcal{C}_1$ and $\mathcal{C}_2$.
Consider the two vertices $v_1\in V(\mathcal{C}_1)$
and $v_2\in V(\mathcal{C}_2)$ such that the distance $d(v_1,v_2)$
between $v_1$ and $v_2$ is maximum. Let $c$ be a $CFC$-coloring of $G$ with
$h(G)$ colors.
Since any $v_1$ -- $v_2$ path in $G$
contains all the edges of $\mathcal{C}_1$ and $\mathcal{C}_2$,
there is no conflict-free path connecting $v_1$ and $v_2$.
Consequently, $h(G)< cfc(G)$.
Together with Theorem~\ref{thm5},
we have $cfc(G)=h(G)+1$ in this case.

Next, we assume that there is only one component of $C(G)$
whose conflict-free connection number
is $h(G)$, say $\mathcal{C}_0$.
Now we give an edge-coloring of $G$. First, we color
$\mathcal{C}_0$ with $h(G)$ colors, say~$1,2,\dots,h(G)$, just like
the coloring of a path stated in Theorem~\ref{thm1} of \cite{Czap}.
Let $e_0$ be the edge colored with color $h(G)$ in $\mathcal{C}_0$.
Similarly, we color all the other components $K$ of $C(G)$
with the colors from $\{1,\dots,h(G)-1\}$.
Note that only $e_0$ is assigned $h(G)$ among all the edges of $C(G)$.

Then according to Lemma~$2$, we choose in any nontrivial block
of $G$ an edge so that all chosen edges form a matching $S$.
We color the edges from $S$ with color $h(G)$,
and the remaining edges with color~$1$.

In the following we have to  show that for any two distinct vertices
$x$ and $y$, there is a conflict-free $x$ -- $y$ path.
If the vertices $x$ and $y$ are from the same component of $C(G)$,
then such a path exists according to Theorem~\ref{thm1}.
If they are in the same nontrivial block,
then by Lemma~$1$, there is an $x$ -- $y$ path going through the edge
assigned $h(G)$.
If none of the above situations appears, then $x$ and $y$
are either from distinct components of $C(G)$,
or distinct nontrivial blocks, or one is from a component of $C(G)$
and the other from a nontrivial block.

Consider a shortest $x$ -- $y$ path $P$ in $G$.
Let $v_1,\dots,v_{r-1}$ be all cut-vertices of $G$
contained in $P$, in this order.
Set $x=v_0$ and $y=v_r$.
The path $P$
goes through blocks $B_1,B_2\dots,B_r$ indicated by the vertices
$v_0$ and $v_1$, $v_1$ and $v_2$, $\dots$, $v_{r-1}$ and $v_r$, respectively.
At least one of the blocks are nontrivial.
If $P$ must go through the edge $e_0$,
then in each block $B_i, 1\leq i\leq r$, we choose
a monochromatic $v_{i-1}$ -- $v_i$ path.
The path concatenated of the above monochromatic paths is a desired one,
since $h(G)$ only appears once.
Otherwise,
we consider the first nontrivial block $B_i,i\in\{1,\dots,r\}$.
In it, we choose a conflict-free $v_{i-1}$ -- $v_i$ path
going through the edge of $B_i$ colored with $h(G)$.
Then in the remaining blocks $B_j,j\in\{1,\dots,r\}\setminus \{i\}$,
we choose a monochromatic $v_{j-1}$ -- $v_j$ path.
The searched conflict-free $x$ -- $y$ path is then concatenated of these
above paths. The resulting $x$ -- $y$ path contains only one edge
assigned $h(G)$. Combining the fact $cfc(G)\geq h(G)$,
we have $cfc(G)=h(G)$ in this case.

Therefore, from above, it is easy to see that there does not exist the case simultaneously satisfying
$cfc(G)=h(G)+1$ and there is only one
component of $C(G)$ whose conflict-free connection number
attains $h(G)$.

In contrast, if $cfc(G)=h(G)$, there is only one
component of $C(G)$ whose conflict-free connection number
attains $h(G)$. Otherwise, $cfc(G)=h(G)+1$.

The result thus follows.
\end{proof}

As a byproduct, we can immediately get the value of
the conflict-free connection number
of a connected claw-free graph.
Before it, we state a structure theorem concerning
a connected claw-free graph.
Notice that a complete graph is claw-free.
Recall that for a connected claw-free graph $G$,
each component of $C(G)$ is a cut-path of $G$.
Let $p$ be the length of a longest
cut-path of $G$.

\begin{thm}\label{thm8}
Let $G$ be a connected claw-free graph.
Then $G$ must belong to one of the following four cases:

\ \ $i)$ $G$ is complete;

\ $ii)$ $G$ is noncomplete and~$2$-edge-connected;

$iii)$ $C(G)$ has at least two components $K$ satisfying $cfc(K)=\lceil log_2 (p+1)\rceil$;

\ $iv)$ $C(G)$ has only one component $K$ satisfying $cfc(K)=\lceil log_2 (p+1)\rceil$.
\end{thm}
\begin{proof} There are two cases according to
whether $G$ has a cut-edge or not.
If $G$ has no cut-edge, we can distinct into two subcases
according to whether $G$ is complete or not.
If $G$ has a cut-edge, then we distinct into two subcases
according to whether $C(G)$ has only one component $K$ satisfying $cfc(K)=\lceil log_2 (p+1)\rceil$ or not.
Thus, a connected claw-free graph $G$
must be in one of the above four subcases.
\end{proof}

According to Lemma~\ref{lem4},
Theorems~\ref{thm3},~\ref{thm7} and \ref{thm8},
we get the following result.

\begin{thm}\label{thm9}
Let $G$ be a connected claw-free graph of order~$n\geq2$.
Then we have

 \ \ $i)$ $cfc(G)=1$ if $G$ is complete;

 \ $ii)$ $cfc(G)=2$ if $G$ is noncomplete and $2$-edge-connected,
 or $p=1$ and $n\geq3$;

$iii)$ $cfc(G)=\lceil log_2 (p+1)\rceil +1$, if $C(G)$
has at least two components $K$ satisfying $cfc(K)=\lceil log_2 (p+1)\rceil$;
otherwise, $cfc(G)=\lceil log_2 (p+1)\rceil$, where $p\geq2$.
\end{thm}

Since line graphs are claw-free, from Theorems~\ref{thm8} and~\ref{thm9} we immediately get the following result.

\begin{cor}\label{cor2}
Let $G$ be a connected line graph of order $n\geq 2$.
Then we have

 \ \ $i)$ $cfc(G)=1$ if $G$ is complete;

 \ $ii)$ $cfc(G)=2$ if $G$ is noncomplete and $2$-edge-connected,
 or $p=1$ and $n\geq3$;

$iii)$ $cfc(G)=\lceil log_2 (p+1)\rceil +1$, if $C(G)$
has at least two components $K$ satisfying $cfc(K)=\lceil log_2 (p+1)\rceil$;
otherwise, $cfc(G)=\lceil log_2 (p+1)\rceil$, where $p\geq 2$.
\end{cor}

Next, for a connected graph $G$ and a positive integer $k$, we compare $cfc(L^{k+1}(G))$ and $cfc(L^k(G))$.
For almost all cases, we find that $cfc(L^{k+1}(G))\leq cfc(L^k(G))$. However, note that if $G$ is a complete graph of order $n\geq 4$,
then $L(G)$ is noncomplete,
since there exist two nonadjacent edges in $G$.
In this case, we have $cfc(L(G))\geq 2 > 1 = cfc(G)$.
So we first characterize the connected graphs whose line graphs
are complete graphs.

\begin{lem}\label{lem5}
The line graph $L(G)$ of a connected graph $G$ is complete if and only if
$G$ is isomorphic to a star or $K_3$.
\end{lem}

\begin{proof} If $G$ is isomorphic to a star
or $K_3$, then obviously $L(G)$ is complete.

Conversely, suppose $L(G)$ is complete. From Whitney
isomorphism theorem of line graphs (see \cite{Line}), i.e., two graphs
$H$ and $H'$ have isomorphic line graphs if and only if
$H$ and $H'$ are isomorphic, or one of them is isomorphic to the claw
$K_{1,3}$ and the other is isomorphic to the triangle $K_3$, we immediately
get that $G$ is isomorphic to a star or $K_3$.
\end{proof}

By Lemma~\ref{lem5} we get the following result.
\begin{thm}
Let $G$ be a connected graph which is not isomorphic
to a star
of order at least~$5$, and $k$ be an arbitrary positive integer.
Then we have $cfc(L^{k+1}(G)) \leq cfc(L^k(G))$.
\end{thm}

\begin{proof}
To the contrary, we suppose that there exists a positive integer $k_0$ such that $cfc(L^{k_0+1}(G))>cfc(L^{k_0}(G))$.
We first claim that $L^{k_0+1}(G)$ has at least one cut-edge.
Otherwise, by Proposition~\ref{pro1} and Theorem~\ref{thm4},
we have $cfc(L^{k_0+1}(G))\leq 2$.
If $L^{k_0}(G)$ is complete, then it follows from Lemma~\ref{lem5}
that $L^{k_0}(G)\cong C_3$.
Then $L^{k_0+1}(G)$
is also complete, implying $cfc(L^{k_0+1}(G))=cfc(L^{k_0}(G))=1$.
If $L^{k_0}(G)$ is noncomplete,
then by Proposition~\ref{pro1}, we have $cfc(L^{k_0}(G))\geq 2$;
clearly, $cfc(L^{k_0+1}(G))\leq cfc(L^{k_0}(G))$ in this case.
In both cases, we have $cfc(L^{k_0+1}(G)) \leq cfc(L^{k_0}(G))$,
a contradiction.

From Corollary~\ref{cor1},
it follows that for a positive integer $i$, each component of $C(L^i(G))$
is a cut-path of $L^i(G)$.
Let $p_i$ be the length
of a largest path of $C(L^i(G))$.
Then we have $p_{i+1}=p_i-1$, meaning
$h(L^{i+1}(G))\leq h(L^i(G))$.
Set $h(L^{k_0}(G))=q$.
Since $L^{k_0+1}(G)$ has a cut-edge,
we deduce that $q\geq2$.
And we know $h(L^{k_0+1}(G))=q-1$ or $h(L^{k_0+1}(G))=q$. If $h(L^{k_0+1}(G))=q-1$, by
Theorem~\ref{thm5}, we have $q-1\leq cfc(L^{k_0+1}(G))\leq q$.
For the same reason, $q\leq cfc(L^{k_0}(G))\leq q+1$.
Thus, it makes a contradiction to the supposition that $cfc(L^{k_0+1}(G))>cfc(L^{k_0}(G))$.

Then we have $h(L^{k_0+1}(G))=q$,
$cfc(L^{k_0+1}(G))=q+1$ and $cfc(L^{k_0}(G))=q$.
By Theorem~\ref{thm7}, there are at least two components
of $C(L^{k_0+1}(G))$ whose conflict-free connection numbers are $q$,
and there is only one component of $C(L^{k_0}(G))$
whose conflict-free connection number is $q$.
Since every cut-path of $L^{k_0+1}(G)$ corresponds
to a cut-path of $L^{k_0}(G)$,
a cut-path of $L^{k_0+1}(G)$
is shorter
than its corresponding cut-path of $L^{k_0}(G)$.
So there is at most one component of $C(L^{k_0+1}(G))$
whose conflict-free connection number is $q$, a contradiction.
Thus, we have $cfc(L^{k+1}(G))\leq cfc(L^k(G))$
for any positive integer $k$.
\end{proof}

If $G$ is a star of order at least~$5$,
then $L^i(G) \ (i\geq 2)$ are noncomplete and~$2$-connected.
The following result is easily obtained according to Theorem~\ref{thm2}.
\begin{thm}
Let $G$ be isomorphic to a star
of order at least~$5$, and $k\geq 2$ be a positive integer.
Then we have $cfc(L^{k+1}(G)) = cfc(L^k(G))$.
\end{thm}

Combining the above two theorems, we get a main result of this section.

\begin{thm}
For an arbitrary connected graph $G$
and an arbitrary positive integer $k$, we always have $cfc(L^{k+1}(G))\leq cfc(L^k(G))$,
with only the exception that $G$ is isomorphic to a star of order at least~$5$
and $k=1$.
\end{thm}

From Theorem~\ref{thm6},
we know the existence of a positive integer $k$ such that $cfc(L^k(G))\leq 2$. From Proposition \ref{pro1} we know that only complete graphs
have the cfc-value equal to 1. So, the iterated line graph $L^k(G)$ of a connected graph $G$ has a cfc-value 1 if and only if $G$ is complete
for $k=0$ from Proposition \ref{pro1}, or $G$ is isomorphic to
a star of order at least~$3$
for $k=1$ from Lemma \ref{lem5}, or $G$ is $K_3$ for all $k\geq 1$, 
or $G$ is a path of order $n\geq 4$ for $k=n-2$.
Next, we want to find the smallest nonnegative
integer $k_0$ such that $cfc(L^{k_0}(G))=2$.
Let $k$ an arbitrary
nonnegative integer.
Based on Proposition~\ref{pro1}, Theorems~\ref{thm1}~through~\ref{thm4},
Lemmas~\ref{lem3}~and~\ref{lem5},
we begin with the investigation of
the exact value of $cfc(L^k(G))$
when $G$ is a path, a complete graph, a star, or
a~noncomplete $2$-edge-connected graph.

\begin{lem}\label{lem6}
Let $n\geq 2$ be a positive integer.
Then $cfc(L^k(P_n))=\lceil log_2(n-k)\rceil$
if $k<n-1$;
otherwise, $cfc(L^k(P_n))=0$.
\end{lem}

\begin{lem}\label{lem7}
Let $G$ be a complete graph of order $n\geq 3$.
Then $cfc(L^k(G))=1$ for any nonnegative integer $k$
if $n=3$;
$cfc(G)=1$ and $cfc(L^k(G))=2$ for any positive integer $k$
if $n\geq 4$.
\end{lem}

\begin{lem}\label{lem8}
Let $G$ be a star of order $n\geq 4$.
Then $cfc(G)=n-1$; $cfc(L(G))=1$;
for a positive integer $k\geq 2$,
$cfc(L^k(G))=1$ if $n=4$,
$cfc(L^k(G))=2$ if $n\geq 5$.
\end{lem}

\begin{lem}\label{lem9}
Let $G$ be a~noncomplete $2$-edge-connected graph of
order $n\geq 4$.
Then $cfc(L^k(G))=2$ for a nonnegative integer $k$.
\end{lem}

Let $\mathcal{G}=\{G \ | \ G$ is a connected graph of order $n\geq4$, $G$ has
a cut-edge, $G$ is not a path or a star$\}$.
Except for the above four kinds of graphs in Lemmas \ref{lem6} through
\ref{lem9},
we know little about the exact values
of the conflict-free connection numbers of other connected graphs,
even for a general tree.
So for a graph $G\in \mathcal{G}$, it is difficult to give the value of
$cfc(L^k(G))$ when $k=0$.
However, based on Corollaries~\ref{cor1}~and~\ref{cor2},
we can give the value of $cfc(L^k(G))$ when $k\geq 1$.
Set $p_0$ be the length of a longest cut-path of $L(G)$,
and let $p_0=0$ if $L(G)$ is~$2$-edge-connected.

\begin{lem}\label{lem10}
Let $G\in \mathcal{G}$ and let $k$ be an arbitrary positive integer.
Then we have

\ $i)$ $cfc(L^k(G))=2$ always holds if $p_0\leq 1$
or there is only one component $K$ of $L(G)$
satisfying $cfc(K)=h(L(G))=2$;

$ii)$ otherwise,
for $k\leq p_0-1$, $cfc(L^k(G))=\lceil log_2(p_0-k+2)\rceil$
if there are only one component $K$ of $L^k(G)$
satisfying $cfc(K)=\lceil log_2(p_0-k+2)\rceil$,
and
$cfc(L^k(G))=\lceil log_2(p_0-k+2)\rceil+1$
if there are at least two components $K$ of $L^k(G)$
satisfying $cfc(K)=\lceil log_2(p_0-k+2)\rceil$;
for $k> p_0-1$, $cfc(L^k(G))=2$ always holds.
\end{lem}

From Lemmas~\ref{lem6} through \ref{lem10},
we can easily get the smallest nonnegative
integer $k_0$ such that $cfc(L^{k_0}(G))=2$.

\begin{thm}
Let $G$ be a connected graph
and $k_0$ be the smallest nonnegative
integer such that $cfc(L^{k_0}(G))=2$.
Then we have

\ \ $i)$ for $G\in \{K_2,K_3,K_{1,3}\}$, $k_0$ does not exist;

\ $ii)$ for a path of order~$3$, $k_0=0$;
        for a path of order $n\geq 4$, $k_0=n-4$;

$iii)$ for a complete graph of order at least~$4$, $k_0=1$;

\ $iv)$ for a star of order at least~$5$, $k_0=2$;

\ \ $v)$ for a~noncomplete $2$-edge-connected graph, $k_0=0$;

\ $vi)$ for a graph $G\in \mathcal{G}$, $k_0=0$ if $cfc(G)=2$;
$k_0=1$ if $C(L(G))=\emptyset$
or $C(L(G))$ is a linear forest whose each component is of order~$2$
or there is only one
component $K$ of $L(G)$ satisfying $cfc(K)=h(L(G))=2$;
otherwise,
$k_0=p_0-2$ if there is only one path of length $p_0$ in $L(G)$
and there is no path of length $p_0-1$ in $L(G)$ with $p_0\geq 4$,
$k_0=p_0-1$ if there is only one path of length $p_0$ in $L(G)$
and there is a path of length $p_0-1$ in $L(G)$ with $p_0\geq 3$,
$k_0=p_0$ if there are at least two paths of length $p_0$ in $L(G)$
with $p_0\geq 2$.
\end{thm}
\begin{proof} Obviously, we can get $i)$ through $v)$
from Lemmas~\ref{lem6}~through~\ref{lem9}.

For $vi)$, if $cfc(G)=2$, then we have $k_0=0$.
If $C(L(G))=\emptyset$, then we have $L(G)$ is~$2$-edge-connected,
and hence, $cfc(L(G))=2$ by Theorem~\ref{thm4};
if $C(L(G))$ is a linear forest whose each component is of order~$2$,
then from Theorem~\ref{thm3} it follows that $cfc(L(G))=2$;
if there is only one
component $K$ of $L(G)$ satisfying $cfc(K)=h(L(G))=2$,
then it follows from Corollary~\ref{cor2} that $cfc(L(G))=2$.
Thus, in the above three cases, we obtain $k_0=1$.
In the following, we consider the case that
both $cfc(G)\geq3$ and $cfc(L(G))\geq3$.
First, we give a fact that
the largest integer $\ell$ such that $cfc(P_\ell)=2$
is~$4$.
Let $G_0$ be a connected graph,
then from Corollary~\ref{cor2}, we have $cfc(G_0)=3$
if there is a component $K$ satisfying $cfc(K)=h(G_0)=3$
or there are at least two components
$K$ satisfying $cfc(K)=h(G_0)=2$.
However, $cfc(G_0)=2$
if there is a path of length~$3$ and there is no
path of length~$2$ in $G_0$,
or if there is a path of length~$2$ and there is a
path of length~$1$ in $G_0$,
or if there are at least two components each of which
is of order~$2$ in $G_0$.
Correspondingly, we get our results.
\end{proof}


\begin{thebibliography}{111}
\bibitem{Andrews} E. Andrews, E. Laforge, C. Lumduanhom, P. Zhang, \emph{On proper-path colorings in graphs}, J. Combin. Math. Combin. Comput. 97(2016), 189--207.

\bibitem{Beineke} L.W. Beineke, \emph{Characterizations of derived graphs}, J. Combin. Theory 9(2)(1970), 129--135.

\bibitem{Bondy} J.A. Bondy, U.S.R. Murty, \emph{Graph Theory with Applications},
The Macmillan Press, London and Basingstoker, 1976.

\bibitem{Magnant} V. Borozan, S. Fujita, A. Gerek, C. Magnant, Y. Manoussakis, L. Montero, Zs. Tuza,
\emph{Proper connection of graphs}, Discrete Math. 312(2012), 2550--2560.

\bibitem{Char1} G. Chartrand, G.L. Johns, K.A. McKeon, P. Zhang,
\emph{Rainbow connection in graphs}, Math. Bohem. 133(1)(2008), 85--98.

\bibitem{Char2} G. Chartrand, M.J. Stewart,
\emph{The connectivity of line-graphs}, Math. Ann. 182(1969), 170--174.

\bibitem{Czap} J. Czap, S. Jendrol', J. Valiska, \emph{Conflict-free connections of graphs}, Discuss. Math. Graph Theory, in press.

\bibitem{Line} R.L. Hemminger, L.W. Beineke, Line graphs and line digraphs, in: L.W. Beineke, R.J. Wilson, Selected Topics in Graph Theory, Academic Press Inc., pp. 271--305 (1978).

\bibitem{Li1} X. Li, C. Magnant, \emph{Properly colored notions of connectivity - a dynamic survey}, Theory \& Appl. Graphs, 0(1)(2015), Art. 2.

\bibitem{LSS} X. Li, Y. Shi, Y. Sun, Rainbow connections of graphs: A survey, {\it Graphs \& Combin.}  29(2013), 1--38.

\bibitem{Li} X. Li, Y. Sun, \emph{Rainbow Connections of Graphs}, New York, SpringerBriefs in Math., Springer, 2012.

\bibitem{Pach} J. Pach, G. Tardos, \emph{Conflict-free colourings of graphs and hypergraphs}, Comb. Probab. Comput. 18(2009), 819--834.

\end{thebibliography}
\end{document}